\documentclass[12pt,a4paper]{amsart}
\usepackage{graphicx,amssymb}
\usepackage{amsmath,amssymb,amscd}
\input xy
\xyoption{all}
\usepackage[all]{xy}
\usepackage{hyperref}

\newcommand{\ra}{\rightarrow}

\newcommand{\cO}{\mathcal{O}}

\theoremstyle{plain}
\newtheorem{theorem}{Theorem}[section]
\newtheorem{lem}[theorem]{Lemma}
\newtheorem{prop}[theorem]{Proposition}
\newtheorem{cor}[theorem]{Corollary}

\newtheorem{rem}[theorem]{Remark}

\numberwithin{equation}{section}
\begin{document}
\title[Clifford]{Lower bounds for Clifford indices in rank three}

\author{H. Lange}
\author{P. E. Newstead}

\address{H. Lange\\Department Mathematik\\
              Universit\"at Erlangen-N\"urnberg\\
              Bismarckstra\ss e $1\frac{ 1}{2}$\\
              D-$91054$ Erlangen\\
              Germany}
              \email{lange@mi.uni-erlangen.de}
\address{P.E. Newstead\\Department of Mathematical Sciences\\
              University of Liverpool\\
              Peach Street, Liverpool L69 7ZL, UK}
\email{newstead@liv.ac.uk}

\thanks{Both authors are members of the research group VBAC (Vector Bundles on Algebraic Curves). The second author 
would like to thank the Department Mathematik der Universit\"at 
         Erlangen-N\"urnberg for its hospitality}

\begin{abstract}
Clifford indices for semistable vector bundles on a smooth projective curve of genus at least 4 were defined
in previous papers of the authors. In the present paper, we establish lower bounds for the Clifford indices 
for rank 3 bundles. As a consequence we show that, on smooth plane curves of degree at least 10, there exist non-generated 
bundles of rank 3 computing one of the Clifford indices.
\end{abstract}
\maketitle

\section{Introduction}

Let $C$ be a smooth projective curve of genus $g \geq 4$ defined over an algebraically closed field of characteristic zero. 
In  \cite{cl}, we generalised the classical Clifford index $\gamma_1$ of $C$ to semistable bundles in the 
following way. First we define, for any vector bundle $E$ of rank $n$ and degree $d$, 
$$
\gamma(E) := \frac{1}{n} \left(d - 2(h^0(E) -n)\right) = \mu(E) -2\frac{h^0(E)}{n} + 2,
$$
where $\mu(E)=\frac{d}n$.
Then 
the Clifford indices $\gamma_n$ and $\gamma_n'$ are defined by 
$$
\gamma_n := \min_{E} \left\{ \gamma(E) \left|  
\begin{array}{c}   E \;\mbox{semistable of rank}\; n \\
h^0(E) \geq n+1,\; \mu(E) \leq g-1
\end{array} \right\} \right.
$$
and
$$
\gamma_n' := \min_{E} \left\{ \gamma(E) \;\left| 
\begin{array}{c} E \;\mbox{semistable of rank}\; n \\
h^0(E) \geq 2n,\; \mu(E) \leq g-1
\end{array} \right\}. \right.
$$
We say that a bundle $E$ {\em contributes} 
to $\gamma_n$ if it is semistable of rank $n$ with $\mu(E) \leq g-1$ and $h^0(E)\ge n+1$ and that $E$ 
{\em computes} $\gamma_n$ if in addition $\gamma(E)=\gamma_n$. Similar definitions are made for $\gamma'_n$. 

In \cite{cl} we obtained a number of results for the Clifford indices, including a useful lower bound for $\gamma_2'$ and 
a formula for $\gamma_2$. In a second paper \cite{cl1} we showed that, under certain conditions, any bundle computing 
$\gamma_n$ or $\gamma_n'$ is generated (a feature assumed in some earlier definitions of E. Ballico 
\cite{b}). These conditions hold when $n=2$ but could fail for $n\ge3$; however, for $n=3$, we were able to show 
that bundles computing either index are generically generated.

Our main object in this paper is to obtain a useful lower bound for $\gamma_3'$; this is almost certainly not best 
possible, but it does enable us to find a lower bound for $\gamma_3$ and in some cases to compute $\gamma_3$. 
To state our results, we recall that the {\em gonality sequence} 
$d_1,d_2,\ldots,d_r,\ldots$ of $C$ is defined by 
$$
d_r := \min \{ d_L \;|\; L \; \mbox{a line bundle on} \; C \; \mbox{with} \; h^0(L) \geq r +1\}.
$$
Our main theorem is \\

\noindent{\bf Theorem \ref{thm3.4}.}
\begin{em} Suppose $\gamma_1\ge3$. Then
$$\gamma_3'\ge\min\left\{\frac{d_9}3-2,\frac{2\gamma_2'+1}3\right\}.$$
\end{em}

\noindent (For $\gamma_1\le2$, we know the precise values of all the Clifford indices (see Remark \ref{rem4.2}).)

As a consequence of this theorem, we obtain a lower bound for $\gamma_3$ when $\frac{d_2}2\ge\frac{d_3}3$ and precise 
values for $\gamma_3$ for general curves of genus $g\ge7$ (Proposition \ref{pr3.8}), for plane curves of degree $\delta\ge10$ 
(Proposition \ref{pr4.6}) and for curves of Clifford dimension at least 3 (see Proposition \ref{pr4.9}).

The proof of the theorem involves two different methods for obtaining a lower bound for $\gamma_3'$. The first method 
(in section 2) is similar to arguments in \cite{cl} and involves considering the dimension of $h^0(F)$ for suitably 
chosen subbundles $F$ of $E$; this works well for $d\le 2g+3$. The second method (in section 3) involves taking a 
proper subbundle $F$ of maximal slope and estimating the Clifford indices of $F$ and $E/F$; this works best when 
$d\ge 2g+4$. The combination of the methods gives the theorem.

Throughout the paper (except in Remark \ref{rem4.2}), $C$ will denote a smooth projective curve of genus $g$ and Clifford 
index $\gamma_1\ge3$ (hence also $g\ge7$) defined 
over an algebraically closed field of characteristic zero. For a vector bundle $G$ on $C$, the rank and 
degree of $G$ will be denoted by $r_G$ and $d_G$ respectively. We shall make use of the following facts about 
the gonality sequence:
$$d_r<d_{r+1} \mbox{ for all }r;\ \ d_r\ge\min\{\gamma_1+2r,g+r-1\}$$
(see \cite[Lemmas 4.2(a) and 4.6]{cl}). The following lemma plays an important role in section 2.\\
 
\noindent{\bf Lemma of Paranjape-Ramanan.}
\begin{em}Let $E$ be a vector bundle on $C$ of rank $n\ge2$ with $h^0(E)=n+s$ ($s\ge1$) such that $E$ possesses 
no proper subbundle $F$ with $h^0(F)>r_F$. Then $d_E\ge  d_{ns}$.\end{em}

\noindent This is a restatement of \cite[Lemma 3.9]{pr} (see \cite[Lemma 4.8]{cl}).

\section{Lower bound for $\gamma'_3$, first method}

Let $E$ be a bundle of degree $d$ computing $\gamma'_3$. 
According to \cite[Theorem 3.3]{cl1}, $E$ is generically generated.
Suppose $h^0(E) = 3 + s$ for some $s \geq 3$.\\

\begin{lem} \label{lem2.1}
If $E$ has no proper subbundle $F$ with $h^0(F) \geq r_F + 1$, then 
\begin{equation} \label{eq2.1}
\gamma(E) \geq \frac{d_9}{3} -2.
\end{equation}
\end{lem}

\begin{proof} By the Lemma of Paranjape-Ramanan, 
$d \geq d_{3s} \geq d_9 + 3s -9$. This gives
$$
\gamma(E) \geq \frac{d_9}{3} + s - 3 - \frac{2s}{3} \geq \frac{d_9}{3} -2.
$$
\end{proof}

\begin{lem}  \label{lem2.2} If $E$ has a line subbundle $F$ with $h^0(F) \geq 2$, 
then
$$
\gamma(E) \geq \min \left\{ \frac{2\gamma_1 + 1}{3}, \frac{1}{3} \left( 4\gamma_1 + 2g + 2 -d \right) \right\}.
$$
\end{lem}

\begin{proof}
Write $h^0(F) = t+1, \; t \geq 1$. Then $d_F \geq d_t$. So $d \geq 3d_t$ which gives
$$
\gamma(E) \geq d_t - \frac{2s}{3} \geq \gamma_1 
$$
if $t \geq \frac{s}{3}$.

If $t < \frac{s}{3}$, then $h^0(E/F) = 2 + v$ with $v > \frac{2s}{3} \geq 2$.
According to \cite[Lemma 7.2]{cl} we have
\begin{equation*}
d_{E/F} \geq \min_{1 \leq u \leq v-1} \left\{ d_{2v}, d_v + \frac{d}{3}, d_u + d_{v-u} \right\}.
\end{equation*} 
There are 3 possibilities:

(i) $d_{E/F} \geq d_{2v}$. This implies 
$$
d \geq d_t + d_{2v} \geq  \gamma_1 + 2t + d_6 + 2v -6 \geq \gamma_1 + 2s + d_6 -6, 
$$
since $v \geq 3$ and $t+v \geq s$, and hence
\begin{equation*} 
\gamma(E) \geq \frac{\gamma_1}{3} + \frac{d_6}{3} - 2 \geq \frac{\gamma_1}{3} + \frac{\gamma_1 + 7}{3} - 2 
= \frac{2\gamma_1 + 1}{3}. 
\end{equation*}

(ii) $d_{E/F} \geq d_v + \frac{d}{3}$. So $d \geq d_t + d_v + \frac{d}{3}$, which is equivalent 
to $\frac{2d}{3} \geq d_t + d_v$. We have $d_t \geq \gamma_1 + 2t$, since $d_t \leq g-1$.

If $d_v < g + v -1$, then $d_v \geq \gamma_1 + 2v$ and hence
$$
2 \gamma(E) \geq 2 \gamma_1 + 2t + 2v - \frac{4s}{3} \geq 2 \gamma_1 + \frac{2s}{3} \geq  2 \gamma_1 +2.
$$

If $d_v \geq g+v-1$, then we have 
$$
\frac{2d}{3} \geq d_t + d_v \geq \gamma_1 + 2t + g + v - 1,
$$ which gives
\begin{eqnarray*}
2 \gamma(E) &\geq & \gamma_1 + t + s + g - 1 - \frac{4s}{3}\\
&=& \gamma_1 + t + g - 1 - \frac{s}{3}\\
& \geq & \gamma_1 + t + \frac{g-1}{2} \geq 2 \gamma_1 + t \geq 2 \gamma_1 + 1.
\end{eqnarray*}
(Here we use Clifford's Theorem for $E$ (see \cite[Theorem 2.1]{bgn}) which implies $s \leq \frac{d}{2} \leq \frac{3(g-1)}{2}$.)
In both cases, $\gamma(E) > \gamma_1$ contradicting the fact that $E$ computes $\gamma'_3$.

(iii) $d_{E/F} \geq d_u + d_{v-u}$. So $d \geq d_t + d_u + d_{v-u}$. 

If $d_u < g+u-1$ and $d_{v-u} < g+v-u-1$, then
$$
d \geq 3 \gamma_1 + 2t + 2u + 2(v-u)  \geq 3 \gamma_1 + 2s.
$$
This gives $\gamma(E) \geq \gamma_1$.

If $d_u \geq g+u-1$ and $d_{v-u} < g + v - u - 1$, then
$$
d \geq 2 \gamma_1 + 2t +g + u - 1 + 2(v-u) \geq 2 \gamma_1 + s + t + v - u + g - 1
$$
and hence, using $\gamma(E) = \frac{d-2s}{3}$,
\begin{eqnarray*}
\gamma(E) &\geq & \frac{2\gamma_1}{3} + \frac{1}{3}(t+v-u+g-1) -\frac{s}{3}\\
& = & \frac{2\gamma_1}{3} + \frac{1}{3}(t+v-u+g-1) 
- \frac{1}{2} \left( \frac{d}{3} - \gamma(E) \right).
\end{eqnarray*}
This gives 
\begin{eqnarray*}
\frac{1}{2}\gamma(E) & \geq & \frac{2\gamma_1}{3} + \frac{1}{3}(t+v-u+g-1) - \frac{d}{6}.
\end{eqnarray*}
So 
\begin{equation*} 
\gamma(E) \geq \frac{4\gamma_1}{3} + \frac{2g-2-d}{3} + \frac{2}{3}(t+v-u) \geq \frac{1}{3}(4\gamma_1 +2g + 2-d).
\end{equation*}

If $d_u < g+u - 1$ and $d_{v-u} \geq g + v - u - 1$, we obtain the same inequality just by interchanging $u$ and $v-u$ in the 
above argument.

If $d_u \geq  g+ u -1$ and $d_{v-u} \geq g+v-u-1$, then
$$
d \geq \gamma_1 + 2t + g +u -1 + g + v - u - 1 \geq \gamma_1 + s + t + 2g - 2.
$$ 
Hence, using $\gamma(E) = \frac{d-2s}{3}$, 
$$
\gamma(E) \geq \frac{\gamma_1}{3} + \frac{t+2g-2}{3} - \frac{s}{3} = \frac{\gamma_1}{3} + \frac{t+2g-2}{3} 
- \frac{1}{2} \left( \frac{d}{3} - \gamma(E) \right)
$$
which gives
\begin{eqnarray*}
\frac{\gamma(E)}{2} &\geq &\frac{\gamma_1}{3} - \frac{d}{6} + \frac{t+2g-2}{3}\\
& \geq & \frac{\gamma_1}{3} - \frac{g-1}{2} + \frac{t + 2g-2}{3} = \frac{\gamma_1}{3} + \frac{g-1}{6} + \frac{t}{3}
\end{eqnarray*}
and hence, since $\gamma_1 \leq \frac{g-1}{2}$,
$$
\gamma(E) \geq \frac{2\gamma_1 +1}{3}.
$$
\end{proof}

\begin{lem} \label{lem2.3}
If $E$ has a subbundle $F$ of rank $2$ with $h^0(F) \geq 3$ and no line subbundle with $h^0 \geq 2$, then 
$$
\gamma(E) \geq \min \left\{ \gamma'_2,  \frac{2 \gamma_1 + 1}{3}, \frac{1}{3}( 2\gamma_1 + 2g + 4 -d) \right\}.
$$
\end{lem}

\begin{proof} 
Write $h^0(F) = 2 + t, \; t \geq 1$. 
Since $E$ is generically generated, we have $h^0(F) \leq 2 + s$. So
$$
1 \leq t \leq s.
$$
The quotient $E/F$ is a line bundle and $h^0(E/F) = 1 + u$ with $u \geq s-t$. 
The Lemma of Paranjape-Ramanan applied to $F$ gives 
$$
d_F \geq d_{2t}.
$$ 

If $\frac{t}{2} > \frac{s}{3}$, \cite[Lemma 2.1]{cl1} gives
\begin{equation} \label{eq4.4}
\gamma(E) > \gamma'_2.
\end{equation}

If $\frac{t}{2} \leq \frac{s}{3}$, then $u \geq s-t \geq \frac{s}{3}$. This implies $u \geq 1$ and 
hence
$$
d_{E/F} \geq d_u.
$$
This gives 
$$
d \geq d_{2t} + d_u.
$$
We distinguish 4 cases:

If $d_{2t} < 2t + g-1$ and $d_u < u +g-1$, then
$$
d \geq 2\gamma_1 + 4t + 2u \geq 2 \gamma_1 + 2s + 2t.
$$
So 
\begin{equation*} 
\gamma(E) \geq \frac{2\gamma_1 + 2t}{3} > \frac{2\gamma_1 + 1}{3}.
\end{equation*} 

If $d_{2t} \geq 2t + g-1$ and $d_u \geq u +g-1$, then 
$$
d \geq 2t +g-1 + u + g-1 \geq s + t + 2g-2.
$$
This implies, using $\gamma(E) = \frac{d-2s}{3}$,
$$
\gamma(E) \geq \frac{t+2g-2}{3} - \frac{s}{3} = \frac{t+2g-2}{3}  - \frac{d-3\gamma(E)}{6}.
$$
Hence, using $d \leq 3g-3$,
$$
\frac{\gamma(E)}{2} \geq \frac{t+2g-2}{3}  - \frac{d}{6} \geq \frac{t}{3} + \frac{g-1}{6} > \frac{g-1}{6} \geq \frac{\gamma_1}{3}.
$$
This gives 
\begin{equation} \label{eq4.5} 
\gamma(E) \geq \frac{2\gamma_1 + 1}{3}
\end{equation}

If $d_{2t} \geq 2t + g-1$ and $d_u < u +g-1$, then 
$$
d \geq \gamma_1 + 2t +g-1 + 2u \geq \gamma_1 + 2s + g-1
$$
and we get as above,
$$
\gamma(E) \geq \frac{\gamma_1}{3} + \frac{g-1}{3} \geq \gamma_1.
$$

Finally, if $d_{2t} < 2t + g-1$ and $d_u \geq u +g-1$, then
$$
d \geq \gamma_1 + 4t + u + g-1 \geq \gamma_1 + s + 3t + g-1.
$$
This gives
$$
\gamma(E) \geq \frac{\gamma_1}{3} -\frac{s}{3} + t + \frac{g-1}{3} = 
\frac{\gamma_1}{3} - \frac{d-3\gamma(E)}{6} + t + \frac{g-1}{3}.
$$
So
$$
\frac{\gamma(E)}{2} \geq \frac{\gamma_1}{3} - \frac{d}{6} + t + \frac{g-1}{3} 
$$
and hence
\begin{equation} \label{eq4.7}
\gamma(E) \geq  \frac{2}{3} \gamma_1 + \frac{2g-d+4}{3}.
\end{equation}
Combining \eqref{eq4.4}, \eqref{eq4.5} and \eqref{eq4.7} we get the result.
\end{proof}

\begin{prop} \label{prop4.1}
Let $E$ be a semistable bundle of degree $d$ computing $\gamma'_3$. Then
$$
\gamma(E) \geq \min \left\{ \frac{d_9}{3} -2, \gamma'_2, \frac{2\gamma_1+1}{3},   
\frac{1}{3}(2\gamma_1 + 2g-d+4) \right\}.
$$ 
\end{prop}

\begin{proof}
This is a consequence of Lemmas \ref{lem2.1}, \ref{lem2.2} amd \ref{lem2.3}.
\end{proof}

\section{Lower bound for $\gamma'_3$, second method}

Let $E$ be a bundle of degree $d$ computing $\gamma'_3$. Then $d \leq 3g-3$ and $h^0(E) = 3 + s$ for some 
$s \geq 3$.

Let $F$ be a proper subbundle of $E$ of maximal slope. By \cite{ms} we have
\begin{equation} \label{ms}
\mu(E/F) - \mu(F) \leq g,
\end{equation}

\begin{lem} \label{lemrk1}
If $F$ has rank $1$, then
$$
\gamma(E) \geq \min \left\{ \frac{\gamma_1 + 2 \gamma'_2}{3}, \frac{2 \gamma'_2}{3} + \frac{d-2g}{9}, 
\frac{\gamma_1}{3} + \frac{2d-6}{9} \right\}.
$$ 
\end{lem}

\begin{proof}
By \eqref{ms} we have $d_{E/F} - 2 d_F \leq 2g$. Since $d_{E/F} + d_F = d$, this gives $d_F \geq \frac{d-2g}{3}$. 
The semistability of $E$ implies $d_F \leq \frac{d}{3}$. So altogether we get
\begin{equation} \label{deg}
\frac{d-2g}{3} \leq d_F \leq\frac{d}{3} \quad \mbox{and} \quad \frac{2d}{3} \leq d_{E/F} \leq \frac{2d+2g}{3}.
\end{equation}

If $h^0(F) \geq 2$, then $F$ contributes to $\gamma_1$ and
\begin{equation} \label{eq5.3}
\gamma(F) \geq \gamma_1.
\end{equation}

If $h^0(F) \leq 1$, then by definition of $\gamma$ and \eqref{deg} we get
\begin{equation} \label{eq5.4}
\gamma(F) \geq d_F \geq \frac{d-2g}{3}.
\end{equation} 

If $E/F$ is not semistable, it has a line subbundle $L$ of degree $> \frac{d}{3}$. Pulling back $L$ to $E$, we obtain a rank-2 
subbundle of $E$ with slope $> d_F$ contradicting the maximality of $\mu(F)$. So $E/F$ is semistable.

Suppose $d_{E/F} \leq 2g-2$. If $h^0(E/F) \geq 4$, then $E/F$ contributes to $\gamma'_2$ which gives
\begin{equation} \label{eq5.5}
\gamma(E/F) \geq \gamma'_2.
\end{equation}

If $h^0(E/F) \leq 3$, the definition of $\gamma$ and \eqref{deg} give 
\begin{equation} \label{eq5.6}
\gamma(E/F) \geq \frac{1}{2}(d_{E/F} -2) \geq \frac{d}{3} - 1.
\end{equation}

If $d_{E/F} > 2g-2$, apply the same argument to $(E/F)^* \otimes K$ to get
$$
\gamma(E/F) = \gamma((E/F)^* \otimes K) \geq \gamma'_2 
$$
if $h^1(E/F) \geq 4$
and, for $h^1(E/F) \leq 3$,
\begin{eqnarray*}
\gamma(E/F) & \geq & \frac{1}{2}(4g-4 -d_{E/F} - 2)\\
& \geq & 2g-3 - \frac{2d+2g}{6} = \frac{5g}{3} -\frac{d}{3} -3  \\
& \geq & \frac{2g}{3} -2 \geq \frac{4\gamma_1 + 2}{3} -2 = \frac{4\gamma_1 -4}{3}.
\end{eqnarray*}
Here we use the inequality $\gamma_1 \leq \frac{g-1}{2}$. Since $\gamma_1 \geq \gamma'_2$ 
and 
$\gamma(E/F)$ is a half-integer for $\gamma_1 = 3$, 
this implies \eqref{eq5.5}.

We have proved that, if $h^0(E/F) \geq 4$ or $d_{E/F} > 2g-2$, then \eqref{eq5.5} is valid. If $d_{E/F} \leq 2g-2$ and $h^1(E/F) \leq 3$, then
\eqref{eq5.6} is valid. 

Now note that
$$ 
\gamma(E) \geq \frac{\gamma(F) + 2 \gamma(E/F)}{3}.
$$
So, combining these results with \eqref{eq5.3} and \eqref{eq5.4} and noting that, if $h^0(F) < 2$, 
then $h^0(E/F) \geq 4$, we get the result.
\end{proof}

\begin{lem} \label{lemrk2}
Let the notations be as at the beginning of the section and suppose $F$ has rank $2$. Then
\begin{eqnarray*}
\gamma(E) & \geq & \min \left\{ \frac{\gamma_1 + 2 \gamma'_2}{3}, \frac{\gamma_1}{3} + \frac{2d -2g -6}{9}, 
\frac{2 \gamma'_2}{3} + \frac{d}{9}, \right. \\
&& \hspace*{1.9cm} \left. \frac{2 \gamma'_2}{3} + \frac{4g - d - 6}{9}, \frac{d+2g-12}{9} \right\}.
\end{eqnarray*}
\end{lem}

\begin{proof}
For $F$ of rank 2, inequality \eqref{ms} says $2d_{E/F} - d_F \leq 2g$. Since $d_{E/F} + d_F = d$ and by semistability 
of $E$, $d_{F} \leq \frac{2d}{3}$, we get
\begin{equation} \label{eq5.7}
\frac{2d-2g}{3} \leq d_F \leq \frac{2d}{3} \quad \mbox{and} \quad \frac{d}{3} \leq d_{E/F} \leq \frac{2g +d}{3}. 
\end{equation}
Certainly $F$ is semistable. So if $h^0(F) \geq 4$, then $F$ contributes to $\gamma'_2$, which gives
\begin{equation*} \label{eq5.8}
\gamma(F) \geq \gamma'_2.
\end{equation*}
If $h^0(F) \leq 3$, then by definition of $\gamma$ and \eqref{eq5.7} we get 
\begin{equation*} \label{eq5.9}
\gamma(F) \geq \frac{1}{2}(d_F -2) \geq \frac{1}{2} \left( \frac{2d-2g}{3} -2 \right) = \frac{d}{3} - \frac{g}{3} -1.
\end{equation*}
If $d_{E/F} \leq g-1$ and $h^0(E/F) \geq 2$, then $E/F$ contributes to $\gamma_1$, which gives
\begin{equation*} \label{eq5.10}
\gamma(E/F) \geq \gamma_1.
\end{equation*}
If $d_{E/F} \leq g-1$ and $h^0(E/F) \leq 1$, then by definition of $\gamma$ and \eqref{eq5.7},
\begin{equation*} \label{eq5.11}
\gamma(E/F) \geq d_{E/F} \geq \frac{d}{3}.
\end{equation*}
If $d_{E/F} > g-1$, we apply the same argument to $(E/F)^* \otimes K$ to give:
$$
\gamma(E/F) = \gamma((E/F)^* \otimes K) \geq \gamma_1 
$$ 
if $h^1(E/F) \geq 2$ and, if $h^1(E/F) \leq 1$,
\begin{eqnarray*}
\gamma(E/F) = \gamma((E/F)^* \otimes K & \geq & d_{(E/F)^* \otimes K} = -d_{E/F} + 2g-2\\
& \geq &2g-2 - \frac{2g + d}{3} = \frac{4g}{3} - \frac{d}{3} -2. \\
\end{eqnarray*}
Now we have
$$
\gamma(E) \geq \frac{2\gamma(F) + \gamma(E/F)}{3}.
$$
Note that if $h^0(F) \leq 3$, then $h^0(E/F) \geq 3$. But
it could still happen in this case that $h^1(E/F) \leq 1$. 
So combining the above results, we obtain the lemma.  
\end{proof}

Combining Lemmas \ref{lemrk1} and \ref{lemrk2} we get

\begin{prop} \label{prop5.3}
$$
\gamma'_3 \geq \min \left\{ \frac{2\gamma_1+1}{3}, \frac{2\gamma'_2}{3} + \frac{d-2g}{9}, \frac{\gamma_1}{3} + \frac{2d-2g-6}{9}, 
\frac{d+2g-12}{9} \right\}.
$$
\end{prop}

\begin{proof}
Let $E$ be a semistable bundle of rank 3 and degree $d$ computing $\gamma'_3$.
Combining Lemmas \ref{lemrk1} and \ref{lemrk2} we see that $\gamma(E)$ is greater than or equal to the minimum of the 5 numbers
$$
\frac{\gamma_1 + 2 \gamma'_2}{3}, \frac{2 \gamma'_2}{3} + \frac{d-2g}{9}, 
\frac{\gamma_1}{3} + \frac{2d-2g-6}{9}, \frac{2\gamma'_2}{3} + \frac{4g -d -6}{9}, 
\frac{d+2g-12}{9}.
$$ 
Now, since $d \leq 3g-3$,
$$
\frac{2\gamma'_2}{3} + \frac{4g-d-6}{9} \geq \frac{2\gamma'_2}{3} + \frac{g-3}{9} \geq \frac{2\gamma'_2}{3} + \frac{d-2g}{9}.
$$
Moreover, by \cite[Theorem 5.2]{cl}, $\gamma'_2 \geq \min \{ \gamma_1, \frac{d_4}{2} - 2 \}$. If $\gamma'_2 < \gamma_1$, 
then $\gamma'_2 \geq \frac{d_4}{2} - 2$. Now $d_4 \geq d_1 + 3 \geq \gamma_1 + 5$. So
$2\gamma'_2 \geq d_4 - 4 \geq \gamma_1 + 1$. So
$$
\frac{\gamma_1 + 2 \gamma'_2}{3} \geq \frac{2 \gamma_1 + 1}{3} .
$$ 
This inequality holds also if $\gamma'_2 = \gamma_1$. The result follows.
\end{proof}

\section{The main theorem}

\begin{theorem} \label{thm3.4}
Suppose $\gamma_1 \geq 3$. Then
$$
\gamma'_3 \geq \min \left\{ \frac{d_9}{3} - 2, \frac{2\gamma'_2+1}{3} \right\}.
$$
\end{theorem}

\begin{proof}
Suppose that $E$ is a semistable bundle of rank 3 and degree $d$ computing $\gamma'_3$.
If $d \leq 2g+3$, the result follows from Proposition \ref{prop4.1}.

If $d \geq 2g + 4$, we use Proposition \ref{prop5.3} and show that the numbers occurring in the statement are 
$\geq \frac{2\gamma'_2+1}{3}$. For the first and second ones this is obvious and we have, using $\gamma_1 \leq \frac{g-1}{2}$,
$$
\frac{\gamma_1}{3} + \frac{2d-2g-6}{9} \geq \frac{\gamma_1}{3} + \frac{2g+2}{9} \geq  
\frac{\gamma_1}{3} + \frac{4\gamma_1 + 4}{9} \ > \frac{2\gamma_1}{3} \geq \frac{2\gamma'_2}{3}. 
$$

Also, if $\gamma_1 \geq 4$,
$$ 
\frac{d+2g-12}{9} \geq \frac{4g-8}{9} \geq \frac{8\gamma_1-4}{9} \geq \frac{2\gamma_1+1}{3} \geq \frac{2\gamma'_2+1}{3}.
$$
If $\gamma_1 =3$ and $d \geq 2g + 5$,
$$ 
\frac{d+2g-12}{9} \geq \frac{4g-7}{9} \geq \frac{8\gamma_1-3}{9} = \frac{2\gamma_1+1}{3} \geq \frac{2\gamma'_2+1}{3}.
$$
If $\gamma_1 = 3$ and $d = 2g+4$, then $\frac{d + 2g -12}{9} = \frac{4g - 8}{9}$, while
$\frac{2 \gamma_1 + 1}{3} = \frac{7}{3}$. For $g = 7$, this does not give the required inequality. However in this case we get 
$\gamma(E) \geq \frac{4g-8}{9} = \frac{20}{9}$. Hence $\gamma(E) \geq \frac{7}{3}$, since $3\gamma(E)$ is an integer.
\end{proof}

\begin{rem} \label{rem4.2}
{\em Theorem \ref{thm3.4} is trivially true for $\gamma_1 = 1$ or 2, since then $\gamma'_2 = \gamma'_3 = \gamma_1$. 
For $\gamma_1 = 0$ the result is false. For $\gamma_1 = 2$, the inequality is strict.
}
\end{rem} 

\begin{cor} \label{corr3.6}
Suppose $\frac{d_2}{2} \geq \frac{d_3}{3}$. Then
$$
\gamma_3 \geq \min \left\{ \frac{2\gamma_1 + 1}{3}, \frac{d_3-2}{3} \right\}. 
$$
\end{cor}

\begin{proof}
By \cite[Theorem 6.1]{cl} we have $\gamma_3 = \min \{ \gamma'_3, \frac{d_3-2}{3} \}$. According to \cite[Theorem 5.2]{cl}, 
$\gamma'_2 \geq \min \{ \gamma_1, \frac{d_4}{2} - 2 \}$. Now by the theorem and using $\frac{d_9}{3} -2 \geq \frac{d_3}{3} > 
\frac{d_3 -2}{3}$ we get the result, noting that $\frac{d_4-3}{3} \geq \frac{d_3 -2}{3}$. 
\end{proof}

\begin{rem} \label{re3.7}
{\em If in addition $d_3 \neq 3d_1$, then there always exist semistable bundles $E$ of rank 3 contributing to $\gamma_3$ 
with $\gamma(E) = \frac{d_3 -2}{3}$ (see \cite[Theorem 4.15 (a)]{cl}).
If $d_3 = 3d_1$, there exist 
semistable bundles $E$ of rank 3 and degree $d_3$ which contribute to $\gamma'_3$; hence $\gamma'_3 < \frac{d_3-2}{3}$.   
}
\end{rem}

\begin{prop} \label{pr3.8}
For a general curve of genus $g \geq 7$,
$$
\gamma_3 = \frac{d_3-2}{3} = \frac{1}{3}\left( g+1 - \left[ \frac{g}{4} \right] \right)
$$
and this value is attained by the bundles $E_L$ given by
$$
0 \ra E_L^* \ra H^0(L) \otimes \cO_C \ra L \ra 0,
$$
where $L$ is a line bundle of degree $d_3$ with $h^0(L) = 4$.
\end{prop}

\begin{proof}
In view of Corollary \ref{corr3.6} and Remark \ref{re3.7} it is sufficient to show 
that $\frac{2\gamma_1+1}{3} \geq \frac{d_3 - 2}{3}$. 
This is a straightforward computation. 
\end{proof}

\begin{prop} \label{pr4.6}
For a smooth plane curve of degree $\delta \geq 7$ we have
$$
\gamma'_3 \geq \frac{2\delta - 7}{3}
$$
and 
$$
\gamma_3 = \frac{1}{3}\left( \left[ \frac{3\delta+1}{2} \right] - 2 \right) \quad \mbox{for} \quad \delta \geq 10.
$$
\end{prop}

\begin{proof}
For a smooth plane curve of degree $\delta$, we have $d_9 = 3 \delta$ and $\gamma'_2 = \gamma_1 = \delta-4$ by 
\cite[Proposition 8.1]{cl}. So Theorem \ref{thm3.4} implies 
$$
\gamma'_3 \geq \min \left\{ \delta -2, \frac{2\delta - 7}{3} \right\} = \frac{2\delta - 7}{3}.
$$
By \cite[Proposition 8.3]{cl},
$$
\gamma_3 = \min \left\{ \gamma'_3, \frac{1}{3}\left( \left[ \frac{3\delta+1}{2} \right] - 2 \right) \right\}.
$$
Now the result follows if we have
$$
\frac{1}{3} \left( \left[ \frac{3\delta + 1}{2} \right] - 2 \right) \leq \frac{2 \delta -7}{3},
$$
which is valid for $\delta \geq 10$.
\end{proof}

\begin{rem} \label{rem4.7}
{\em
For $7 \leq \delta \leq 9$ the estimate on $\gamma'_3$ is not sufficient to give the result for $\gamma_3$. For $\delta = 5$ 
(and similarly for $\delta = 6$) the result for $\gamma_3$ is false, since $\gamma_3 = \gamma'_3 = 1$ 
(see \cite[Proposition 2.6 and Theorem 3.6(e)]{cl}). 
}
\end{rem}

We are now in a position to justify the claim made in \cite[Remark 3.8]{cl1}.

\begin{cor} \label{cor4.8}
For a smooth plane curve of degree $\delta \geq 10$, $\gamma_3 > \gamma_2$. Moreover, there exists a non-generated 
bundle computing $\gamma_3$.
\end{cor}

\begin{proof}
We have formulae for both $\gamma_2$ and $\gamma_3$; a straightforward computation shows that $\gamma_3 > \gamma_2$.
The existence statement follows from \cite[Proposition 3.5 and Remark 3.6]{cl1}.
\end{proof}

Finally we consider curves of Clifford dimension at least 3.

\begin{prop} \label{pr4.9}
Let $C$ be a curve for which neither $d_1$ nor $d_2$ computes $\gamma_1$. Then
$$
\gamma_3 = \frac{d_3 -2}{3}.
$$
\end{prop}

\begin{proof}
Suppose $d_r$ computes $\gamma_1$, $r \geq 3$. By \cite[Corollary 3.5]{elms}, $d_r \geq 4r-3 \geq 3r$.
So 
$$
d_3 - 2 \leq d_r - r + 1 \leq 2d_r - 4r + 1 = 2\gamma_1 + 1.
$$
Now $d_3 \leq d_3 - r + 3$ and, since $d_2$ does not compute $\gamma_1$,
$$
d_2 \geq \gamma_1 + 5 = d_r -2r + 5;
$$
hence $\frac{d_2}{2} > \frac{d_3}{3}$ and by
Corollary \ref{corr3.6} we have
$$
\gamma_3 \geq \frac{d_3-2}{3}.
$$
This value is attained by $E_L$ where $L$ is a line bundle of degree $d_3$ with $h^0(L) = 4$.  
\end{proof}

\end{document}